\documentclass[11pt]{article}

\usepackage[a4paper,margin=29mm]{geometry}
\usepackage{amsmath,amssymb,amsthm,mathtools}
\usepackage{microtype}
\usepackage[hidelinks]{hyperref}
\usepackage{enumitem}
\usepackage{booktabs}
\usepackage{array}

\newtheorem{theorem}{Theorem}[section]
\newtheorem{proposition}[theorem]{Proposition}
\newtheorem{lemma}[theorem]{Lemma}
\newtheorem{corollary}[theorem]{Corollary}
\newtheorem{remark}[theorem]{Remark}
\newtheorem{definition}[theorem]{Definition}

\newcommand{\PP}{\mathbf P}
\newcommand{\CC}{\mathbf C}
\newcommand{\OO}{\mathcal O}
\newcommand{\cE}{\mathcal E}
\newcommand{\cI}{\mathcal I}
\newcommand{\cA}{\mathcal A}
\newcommand{\cH}{\mathcal H}
\newcommand{\NS}{\operatorname{NS}}
\newcommand{\Bl}{\operatorname{Bl}}
\newcommand{\Sing}{\operatorname{Sing}}
\newcommand{\ch}{\operatorname{ch}}
\newcommand{\sig}{\sigma}
\newcommand{\uc}{\operatorname{uc}}

\title{Smooth Counterexamples to the Eisenbud--Schreyer--Weyman Ulrich Existence Problem}
\author{Cristian Anghel\\
\small Simion Stoilow Institute of Mathematics of the Romanian Academy\\
\small 21 Calea Grivitei Street, 010702 Bucharest, Romania\\
\small \texttt{cristian.anghel@imar.ro}}
\date{}

\begin{document}
\maketitle

\begin{abstract}
In their 2003 article in the \emph{Journal of the American Mathematical Society},
Eisenbud, Schreyer and Weyman asked whether every embedded projective variety carries
an Ulrich sheaf.  In 2017 Beauville proposed a numerical route toward a surface with no
Ulrich bundles: in Picard rank one, existence forces
$H^2\ge K_S^2-8\chi(\OO_S)$, suggesting a search near the
Bogomolov--Miyaoka--Yau boundary.  We show that the Picard-rank-one hypothesis is not
needed for the obstruction: a rank-independent Bogomolov--Hodge argument gives the
same inequality on every smooth polarized surface.  Using additional
N\'eron--Severi directions on Hirzebruch--Kummer resolutions, the exponent-$3$ Hesse
surface admits a very ample class $H=4A-E$ with
$H^2=7\cdot3^9<16\cdot3^9=K_Y^2-8\chi(\OO_Y)$, hence a smooth counterexample to
the Eisenbud--Schreyer--Weyman problem in its standard formulation.  Consequently its Chow form has no
ESW-type linear determinantal representation arising from an Ulrich sheaf on the embedded surface.  Moreover,
for every $n\ge3$ the Hesse pair $(Y_n,4A_n-E_n)$ is a counterexample, the surfaces are
pairwise non-isomorphic, and $H_n^2/\sig(Y_n)=7/(3n^2-11)\to0$, while
$K_{Y_n}^2/\chi(\OO_{Y_n})\to60/7\approx8.5714$.  A general
arrangement-theoretic Rees-algebra/Segre mechanism yields further infinite families.
\end{abstract}



\section{Introduction}

Throughout the paper the base field is $\CC$, and all varieties and surfaces are
projective unless explicitly stated otherwise.

Eisenbud, Schreyer and Weyman asked whether every embedded projective variety
$X\subset\PP^N$ carries an Ulrich sheaf \cite[p.~543]{ESW}.  We use the standard
scheme-theoretic formulation of this question: an Ulrich sheaf on the embedded
scheme $X$ is a coherent $\OO_X$-module whose support is all of $X$, Ulrich with
respect to $\OO_X(1)$.  This is exactly the formulation used in the modern
restatement of the ESW problem by Hanselka--Kummer
\cite[Proposition~4.2 and Problem~4.3]{HanselkaKummer}.  On a smooth surface such
a full-support Ulrich sheaf is locally free; see Lemma~\ref{lem:full-support-bundle}
below.  Hence the smooth case asks whether every fixed projective embedding carries
an Ulrich vector bundle.  We show that the answer is negative already for smooth
complex projective surfaces.  We do not address the different ambient question of
Ulrich sheaves on nonreduced thickenings having $X$ as their reduced support.

The emphasis on a \emph{fixed} embedding is essential.  On the positive side,
Coskun and Huizenga proved that if $X$ is any smooth complex projective surface
and $H$ is ample, then $(X,mH)$ carries rank-two Ulrich bundles for all
sufficiently large $m$ \cite[Theorem~1.2]{CoskunHuizenga}.  Thus on surfaces one
can always obtain Ulrich bundles after making a polarization sufficiently
positive.  The Eisenbud--Schreyer question is sharper: once a very ample class
$H$---equivalently, a projective embedding---is fixed, must an Ulrich object
exist?  Our examples show that it need not.

This fixed-embedding formulation is also where the original determinantal
motivation of ESW lives.  From the linear resolution of a rank-$r$ Ulrich sheaf,
their construction produces a matrix whose entries are linear forms in Pl\"ucker
coordinates and whose determinant is the $r$-th power of the Chow form of the
embedded variety \cite[Section~3]{ESW}.  Consequently, for each embedding produced
below, the Chow form admits no linear determinantal representation arising from an
Ulrich sheaf on the embedded surface by the ESW construction.  Re-embedding the same abstract surface changes
the Chow form, so asymptotic existence for $(X,mH)$ does not weaken this conclusion.

Following the terminology of Bl\"aser--Eisenbud--Schreyer \cite{BESUlrichComplexity}, for a smooth polarized variety $(X,H)$ we write
\[
 \uc(X,H):=\min\{r\ge1:\text{$(X,H)$ carries an Ulrich bundle of rank $r$}\},
\]
and set $\uc(X,H)=\infty$ if no such bundle exists.  Our first main theorem is
an explicit smooth counterexample.

\begin{theorem}[Smooth counterexample]\label{thm:intro-counterexample}
Let $Y=Y_3(\cH)$ be the minimal resolution of the exponent-$3$
Hirzebruch--Kummer cover associated with the Hesse arrangement of twelve lines.
Writing $A$ for the pullback of the hyperplane class of the projective Kummer
model and $E$ for the reduced exceptional divisor, the class
\[
        H=4A-E
\]
is very ample and satisfies
\[
        H^2=7\cdot3^9
        <16\cdot3^9
        =K_Y^2-8\chi(\OO_Y).
\]
Consequently
\[
        \uc(Y,H)=\infty,
\]
and the embedding defined by $|H|$ supports no Ulrich sheaf with support $Y$.
Moreover $Y$ is a compact ball quotient.
\end{theorem}

The underlying Hirzebruch--Kummer surface is classical in origin; the new input
in Theorem~\ref{thm:intro-counterexample} is the polarization.  More precisely,
we exhibit on this classical Kummer geometry the very ample off-ray class
$4A-E$ and show that it crosses the rank-independent Ulrich threshold.

The proof has three main ingredients: the Rees-algebra/Segre argument proves that
$H=4A-E$ is very ample, the Kummer computation gives
$H^2=7\cdot3^9<16\cdot3^9=\sig(Y)$, and Proposition~\ref{prop:obstruction}
excludes Ulrich bundles of every rank.  The rest of the paper establishes these
three inputs and then explains why the construction persists in families.

The numerical obstruction behind Theorem~\ref{thm:intro-counterexample} is very
simple.  If $\cE$ is a rank-$r$ Ulrich bundle on a smooth polarized surface
$(S,H)$ and
\[
        \alpha=c_1(\cE)-\frac r2(K_S+3H),
\]
then Riemann--Roch gives $\alpha\cdot H=0$ and
\[
 \Delta(\cE)=\alpha^2+\frac{r^2}{4}
 \bigl(H^2-K_S^2+8\chi(\OO_S)\bigr).
\]
Semistability and Bogomolov's inequality give $\Delta(\cE)\ge0$, whereas the Hodge index theorem
gives $\alpha^2\le0$.  Hence every Ulrich bundle forces
\begin{equation}\label{eq:intro-obstruction}
        H^2\ge K_S^2-8\chi(\OO_S)=\sig(S).
\end{equation}
The search for the present examples was motivated by Beauville's 2017 discussion
\cite[slides~18--20]{Beauville2017}; see also \cite{Beauville2018}.  His route to
a surface without Ulrich bundles uses Picard rank one as the device that pins down
$c_1(\cE)=\tfrac r2(K_S+3H)$ numerically and hence yields the inequality
$H^2\ge\sig(S)$.  This led him to ask for a Picard-rank-one surface below that
threshold and to single out the geography strip
\[
        \operatorname{rk}\NS(S)=1,\qquad
        8\chi(\OO_S)<K_S^2<9\chi(\OO_S).
\]
Beauville also observed that there are many surfaces with positive signature but that
``most of them have $\operatorname{rk}\NS(S)>1$'', and concluded,
``It seems hard to get a counter-example out of this''
\cite[slides~19--20]{Beauville2017}.  The construction below enters precisely
this higher-Picard-rank regime and removes the Picard-rank-one restriction from
the numerical obstruction.
The rank-one hypothesis is therefore part of that proof mechanism, not of the
non-existence phenomenon itself.  Indeed, the orthogonal class $\alpha$ above shows
that the same obstruction survives in arbitrary Picard rank because
$\alpha^2\le0$.  Our construction then uses the additional N\'eron--Severi direction
to place a very ample polarization strictly below the signature threshold.  The
stronger Picard-rank-one refinement of Beauville's search remains open.

This also explains geometrically why the Hesse example succeeds.  On a compact
ball quotient with $K_Y=3L$ one has
\[
        L^2=\sig(Y).
\]
When the N\'eron--Severi group has rank one, the polarization is forced to stay
on the canonical numerical ray, so the natural class $L=K_Y/3$ is exactly on the
boundary of \eqref{eq:intro-obstruction}.  Our construction uses the extra
N\'eron--Severi direction supplied by the Kummer model.  For the Hesse member,
\[
        L=5A-E,\qquad H=4A-E=L-A.
\]
Beauville's rank-one route asks whether the canonical root $L$ can be divided
further while remaining very ample.  Our departure can be summarized by the
single off-ray move
\[
             \boxed{\quad L=K_Y/3\ \longmapsto\ H=L-A.\quad}
\]
Thus we do not divide $L$ further; instead we move \emph{off the canonical ray}
by subtracting the nef class $A$, while the Rees construction proves that $H$
remains very ample.  Numerically,
\[
        H^2-L^2=A^2-2L\cdot A<0,
\]
which moves the polarization strictly below the Ulrich threshold.

The inequality $H^2\ge\sig(S)$ itself is a short consequence of standard tools and
may well be familiar to experts in special forms.  We do not rely on a novelty claim
for this inequality.  The new geometric input is the construction of explicit smooth
polarized pairs $(Y,H)$ for which the inequality fails, together with a mechanism that
produces infinitely many such pairs.

Our second main result is that the counterexample is not isolated: the Hesse
construction produces infinitely many pairwise non-isomorphic smooth counterexamples
to the ESW problem.

\begin{theorem}[Infinite Hesse family of counterexamples]\label{thm:intro-hesse-family}
For every integer $n\ge3$, let $Y_n$ be the minimal resolution of the exponent-$n$
Hirzebruch--Kummer cover associated with the Hesse arrangement.  Then
\[
        H_n=4A_n-E_n
\]
is very ample,
\[
        H_n^2=7n^9,
        \qquad
        \sig(Y_n)=(3n^2-11)n^9,
\]
and
\[
        \uc(Y_n,H_n)=\infty.
\]
In particular
\[
        \frac{H_n^2}{\sig(Y_n)}
        =\frac7{3n^2-11}\xrightarrow[n\to\infty]{}0.
\]
The surfaces $Y_n$ are pairwise non-isomorphic.  Thus every embedding defined by
$|H_n|$ is a smooth counterexample to the ESW problem, and these counterexamples are
pairwise non-isomorphic as abstract surfaces.
\end{theorem}

The infinite Hesse family is itself an instance of a more general
arrangement-theoretic mechanism.  Given a line arrangement $\cA$ with $k$ lines and incidence numbers
$t_r$, set
\[
        Q_{\cA}=3-k+\sum_{r\ge2}(r-2)t_r,
        \qquad
        T=\sum_{r\ge3}t_r.
\]
The projective Kummer model has a reduced singular ideal whose low-degree
square-free monomials are controlled by the incidence hypergraph of $\cA$.
When these monomials generate the singular ideal in a fixed degree $\nu$, the
Rees algebra gives a very ample class $(\nu+1)A_n-E_n$ on the resolution for
every exponent $n$.  At the same time the signature grows two powers of $n$
faster than the normalized square of this polarization.

\begin{theorem}[Kummer-family mechanism]\label{thm:intro-mechanism}
Let $\cA$ be a line arrangement admitting a degree-$\nu$ generation datum in the
sense of Section~\ref{sec:rees}.  Put $a=\nu+1$ and let $R_{\cA}$ be as in
\eqref{eq:QR}.  Then
\[
        H_n=aA_n-E_n
\]
is very ample for every $n\ge2$, and
\[
        H_n^2=(a^2-T)n^{k-3},
        \qquad
        \sig(Y_n)=\frac{n^{k-3}}3
        \bigl(Q_{\cA}n^2+R_{\cA}\bigr).
\]
If $Q_{\cA}>0$, then $\uc(Y_n,H_n)=\infty$ for all sufficiently large $n$ and
\[
        \frac{H_n^2}{\sig(Y_n)}\longrightarrow0.
\]
In particular, every such sufficiently large $n$ gives a smooth counterexample to
the ESW problem.
\end{theorem}

Besides Hesse, we treat two companion arrangements.  The dual Hesse arrangement
gives counterexamples for every $n\ge5$.  The complete quadrangle (also called the complete quadrilateral in the
line-arrangement literature) gives counterexamples for every $n\ge6$, while at $n=5$ one has
\[
        K_{Y_5}=3H_5,
        \qquad
        H_5^2=\sig(Y_5)=625,
\]
which realizes, in higher Picard rank, the numerical boundary configuration
$K=3H$, $H^2=\sig$ appearing in Beauville's discussion.

The proofs are organized as follows.  Section~\ref{sec:ulrich} proves the
Picard-rank-free obstruction \eqref{eq:intro-obstruction}.  Sections
\ref{sec:kummer} and \ref{sec:rees} develop the Kummer and Rees-algebra geometry.
Section~\ref{sec:hesse} proves the smooth counterexample and the Hesse-family theorem;
Theorem~\ref{thm:intro-hesse-family} follows from Theorem~\ref{thm:hesse-family}
together with the growth calculation \eqref{eq:hesse-K2-growth}.  Section~\ref{sec:mechanism}
proves Theorem~\ref{thm:intro-mechanism} in the more precise form of
Theorem~\ref{thm:kummer-family}, and Section~\ref{sec:companions} treats the two
companion families.  Finally, Section~\ref{sec:scope} places the result relative
to known existence theorems and to the local non-existence phenomenon for Ulrich modules.

\section{A numerical obstruction for polarized surfaces}\label{sec:ulrich}

We use the convention
\[
        \Delta(\cE)=2r c_2(\cE)-(r-1)c_1(\cE)^2
\]
for a rank-$r$ vector bundle $\cE$.

\begin{proposition}\label{prop:obstruction}
Let $(S,H)$ be a smooth complex projective surface with $H$ very ample.  If
$S$ admits an Ulrich bundle with respect to $H$, then
\begin{equation}\label{eq:obstruction}
        H^2\ge K_S^2-8\chi(\OO_S)=\sig(S).
\end{equation}
\end{proposition}

\begin{proof}
Let $\cE$ be an Ulrich bundle of rank $r$.  By the standard Hilbert-polynomial
characterization of Ulrich bundles \cite[Proposition~2.1]{ESW}, its Hilbert polynomial is
\[
        \chi(\cE(tH))=\frac{rH^2}{2}(t+1)(t+2).
\]
Riemann--Roch gives
\begin{align}
 c_1(\cE)\cdot H&=\frac r2(K_S+3H)\cdot H,\label{eq:c1H}\\
 \ch_2(\cE)&=\frac12 K_S\cdot c_1(\cE)
             +r\bigl(H^2-\chi(\OO_S)\bigr).\label{eq:ch2}
\end{align}
Set
\[
        \alpha=c_1(\cE)-\frac r2(K_S+3H)\in\NS(S)_{\mathbf R}.
\]
Then \eqref{eq:c1H} gives $\alpha\cdot H=0$.  Since
\[
        \Delta(\cE)=c_1(\cE)^2-2r\ch_2(\cE),
\]
substitution of \eqref{eq:ch2} yields
\begin{equation}\label{eq:Delta-shifted}
 \Delta(\cE)=\alpha^2+\frac{r^2}{4}
 \bigl(H^2-K_S^2+8\chi(\OO_S)\bigr).
\end{equation}
Ulrich bundles are Gieseker semistable with respect to $H$
\cite[Theorem~2.9]{CH}, hence $\mu_H$-semistable.  Bogomolov's inequality for
$\mu_H$-semistable torsion-free sheaves on a smooth complex projective surface
\cite[Theorem~3.4.1]{HuybrechtsLehn} therefore gives
$\Delta(\cE)\ge0$.  On the other hand, Hodge index and $\alpha\cdot H=0$ give
$\alpha^2\le0$.  Thus
\[
 0\le\Delta(\cE)
 \le\frac{r^2}{4}
 \bigl(H^2-K_S^2+8\chi(\OO_S)\bigr),
\]
which proves \eqref{eq:obstruction}.
\end{proof}

\begin{remark}[Related Hodge--Bogomolov bounds]\label{rem:bordoni-bounds}
Related Hodge-index and Bogomolov bounds for the same Chern data of Ulrich
bundles appear in Bordoni \cite[Remark~4.1]{Bordoni2025}, in connection with
genus estimates for associated curves.  The shifted class $\alpha$ above is the
form needed here to isolate the rank-independent inequality
$H^2\ge K_S^2-8\chi(\OO_S)$.
\end{remark}

\begin{remark}\label{rem:signature}
By Noether's formula and the Hirzebruch signature theorem,
\[
 K_S^2-8\chi(\OO_S)
   =\frac{K_S^2-2c_2(S)}3
   =\sig(S),
\]
the topological signature of the underlying oriented four-manifold.
\end{remark}

\begin{remark}[Relation with Beauville]\label{rem:beauville-obstruction}
When $\operatorname{rk}\NS(S)=1$, the relation $\alpha\cdot H=0$ forces
$\alpha\equiv0$ in the one-dimensional real N\'eron--Severi space.  Thus the class
$\alpha$ above vanishes numerically, recovering the Picard-rank-one argument emphasized by Beauville
\cite{Beauville2017,Beauville2018}.  In higher Picard rank the orthogonal term
$\alpha^2\le0$ only strengthens the obstruction.
\end{remark}

\begin{proposition}[Ball-quotient displacement]\label{prop:ball-displacement}
Let $Y$ be a smooth compact ball quotient with $K_Y=3L$.  If $A$ is a divisor
such that $H=L-A$ is very ample and
\[
        2L\cdot A>A^2,
\]
then $(Y,H)$ carries no Ulrich bundle.
\end{proposition}

\begin{proof}
For a ball quotient, $K_Y^2=3c_2(Y)$ and hence
$\chi(\OO_Y)=K_Y^2/9$.  Therefore
\[
        \sig(Y)=K_Y^2-8\chi(\OO_Y)=K_Y^2/9=L^2.
\]
But
\[
        H^2-L^2=(L-A)^2-L^2=A^2-2L\cdot A<0.
\]
Proposition~\ref{prop:obstruction} applies.
\end{proof}

\section{Hirzebruch--Kummer covers of line arrangements}\label{sec:kummer}

The Hirzebruch--Kummer construction is classical.  We record the projective model,
local resolution, and numerical invariants here for completeness and in the
normalization needed later for the polarization and Rees-algebra arguments.

Let
\[
        \cA=\{L_1,\ldots,L_k\}
\]
be an arrangement of $k\ge4$ distinct lines in $\PP^2$, not all through one
point.  Let $\ell_i$ be defining linear forms.  For $r\ge2$, let $t_r$ denote
the number of points at which exactly $r$ arrangement lines meet, and set
\[
 f_0=\sum_{r\ge2}t_r,\qquad
 f_1=\sum_{r\ge2}rt_r,\qquad
 T=\sum_{r\ge3}t_r.
\]
We repeatedly use the incidence identity
\begin{equation}\label{eq:pair-identity}
        \sum_{r\ge2}\binom r2 t_r=\binom k2.
\end{equation}

Because the lines are not all concurrent, the forms $\ell_i$ span
$H^0(\PP^2,\OO_{\PP^2}(1))$ and define a linear embedding
\[
 \iota:\PP^2\longrightarrow\PP^{k-1},
 \qquad
 p\longmapsto[\ell_1(p):\cdots:\ell_k(p)]
\]
with image a plane $\Lambda$.  For $n\ge2$ let
\[
 \Phi_n:\PP^{k-1}\longrightarrow\PP^{k-1},
 \qquad [z_1:\cdots:z_k]\longmapsto[z_1^n:\cdots:z_k^n],
\]
and define
\[
        X_n:=\Phi_n^{-1}(\Lambda),
        \qquad
        \pi_n:=\Phi_n|_{X_n}:X_n\to\Lambda\simeq\PP^2.
\]

\subsection{The projective model}

\begin{lemma}\label{lem:s1-generic-field}
Let $Z$ be a Noetherian scheme satisfying $S_1$.  If $Z$ has a unique minimal
point and its local ring at that point is a field, then $Z$ is integral.
\end{lemma}

\begin{proof}
The generic local ring being a field says that $Z$ is generically reduced, i.e.
$R_0$ holds at its unique minimal point.  The standard $(R_0)+(S_1)$ criterion
therefore gives reducedness, while uniqueness of the minimal point gives
irreducibility.
\end{proof}

\begin{proposition}\label{prop:model}
The surface $X_n$ is an integral normal complete intersection of $k-3$
hypersurfaces of degree $n$ in $\PP^{k-1}$.  The map $\pi_n$ is finite of
degree $n^{k-1}$ and
\[
        \pi_n^*\OO_{\PP^2}(1)=\OO_{X_n}(n).
\]
If $A_n$ denotes the pullback of $\OO_{X_n}(1)$ to a resolution of $X_n$, then
\begin{equation}\label{eq:A-square}
        A_n^2=n^{k-3}.
\end{equation}
\end{proposition}

\begin{proof}
The ideal of $\Lambda$ is generated by $k-3$ independent linear forms, whose
pullbacks under $\Phi_n$ are $k-3$ forms of degree $n$.  The inverse image $X_n$
is nonempty, and since $\Phi_n$ is finite it has pure dimension $2$.  Thus these
$k-3$ equations have codimension $k-3$ in the regular scheme $\PP^{k-1}$; hence
they form a regular sequence.

Let $K=\CC(\PP^2)$.  At the generic point, the extension is obtained by adjoining
$n$-th roots of
\[
        \ell_2/\ell_1,\ldots,\ell_k/\ell_1.
\]
Their classes are independent in $K^\times/K^{\times n}$: a relation
\[
 \prod_{i=2}^k(\ell_i/\ell_1)^{a_i}=f^n,
 \qquad 0\le a_i<n,
\]
has valuation $a_j$ along $L_j$ for $j\ge2$, forcing every $a_j=0$.  Thus the
generic algebra is a field of degree $n^{k-1}$.  Since $X_n$ is a complete intersection, it is Cohen--Macaulay and hence satisfies
$S_1$.  The generic algebra is a field, so Lemma~\ref{lem:s1-generic-field}
shows that $X_n$ is integral.

The local analysis in Proposition~\ref{prop:local} below shows that singularities
are isolated, so $X_n$ is regular in codimension one.  The local calculation
used there depends only on the defining complete-intersection equations and the
formal implicit-function theorem; it does not use normality or reducedness of
$X_n$.  Thus this reference is logically non-circular.  Since a complete
intersection is $S_2$, Serre's criterion gives normality.

Finally, $\pi_n^*\OO(1)=\OO_{X_n}(n)$ and $\deg\pi_n=n^{k-1}$, so
\[
        (nA_n)^2=n^{k-1},
\]
which is \eqref{eq:A-square}.
\end{proof}

\subsection{Local singularities and the blow-up resolution}

For a point $p\in\PP^2$ let $r(p)$ be the number of arrangement lines through
$p$.

\begin{lemma}[Blowing up a completed standard cone]\label{lem:completed-cone-blowup}
Let $(R,\mathfrak m)$ be a Noetherian local $\CC$-algebra.  Suppose that its
$\mathfrak m$-adic completion is isomorphic to the completion at the vertex of a
standard graded affine cone
\[
        A=\CC[x_1,\ldots,x_r]/I,
        \qquad C=\operatorname{Proj}A,
\]
where $A$ is a normal domain and $C$ is smooth, and suppose that the isomorphism carries
$\mathfrak m\widehat R$ to the vertex ideal $(x_1,\ldots,x_r)\widehat A$.
Then $\Bl_{\mathfrak m}\operatorname{Spec}R$ is regular along its exceptional
fiber.  Scheme-theoretically that fiber is a reduced copy of $C$, and, if it is
denoted by $E$, then
\[
        \OO(-E)|_E\simeq\OO_C(1).
\]
In particular $N_{E/\Bl_{\mathfrak m}\operatorname{Spec}R}\simeq\OO_C(-1)$.
\end{lemma}

\begin{proof}
The completion map $R\to\widehat R$ is faithfully flat
\cite[Lemma~10.97.3, Tag~00MC]{StacksCompletionFF}.  Flat base change identifies
the Rees algebra of $\mathfrak m$ after completion with the Rees algebra of
$\mathfrak m\widehat R$; equivalently, blow-up commutes with this base change
\cite[Lemma~31.33.3, Tag~0805]{StacksBlowup}.  The blow-up of the vertex of a
standard affine cone over a smooth projective scheme is the total space of
$\OO_C(-1)$, hence is smooth.  Regularity therefore descends from the completed
blow-up along the faithfully flat local maps
\cite[Lemma~10.164.4, Tag~07NG]{StacksRegularDescent}.

The exceptional fiber of the original blow-up is
\[
        \operatorname{Proj}\operatorname{gr}_{\mathfrak m}R.
\]
For every $d$, completion induces
$\mathfrak m^d/\mathfrak m^{d+1}\simeq
(\mathfrak m\widehat R)^d/(\mathfrak m\widehat R)^{d+1}$, so the associated
graded ring is unchanged by completion.  Under the assumed cone description,
the latter associated graded ring is the standard graded ring $A$ itself.
Thus the exceptional fiber is $\operatorname{Proj}A=C$ scheme-theoretically;
it is reduced because $C$ is smooth.  Finally, the tautological invertible sheaf
of the blow-up is $\mathfrak m\OO_{\Bl}=\OO(-E)$, and its restriction to the
exceptional fiber is $\OO_C(1)$.  This proves the assertions.
\end{proof}

\begin{proposition}\label{prop:local}
Let $p$ be an $r$-fold point of $\cA$.
\begin{enumerate}[label=\textup{(\roman*)}]
\item The fiber of $\pi_n$ over $p$ has $n^{k-r-1}$ points.
\item If $r\le2$, $X_n$ is smooth over $p$.
\item If $r\ge3$, write $T_p=\{i:L_i\ni p\}$.  At every
      $q\in\pi_n^{-1}(p)$ there is an isomorphism of complete local
      $\CC$-algebras
      \begin{equation}\label{eq:completed-local-cone}
      \widehat\OO_{X_n,q}\simeq
      \CC[[x_i:i\in T_p]]/(F_1,\ldots,F_{r-2}),
      \end{equation}
      where each $F_a$ is a homogeneous form
      $F_a=\sum_{i\in T_p}\lambda_i^{(a)}x_i^n$, and the ambient coordinate
      $z_i$ maps to $x_i$ for every $i\in T_p$.  The right-hand side is the
      completed local ring at the vertex of the affine cone over the smooth
      complete-intersection curve
      \[
        C_{n,r}=\Phi_n^{-1}(\PP^1)\subset\PP^{r-1},
      \]
      cut out by $r-2$ hypersurfaces of degree $n$.  This curve has
      \[
        d_{n,r}:=\deg C_{n,r}=n^{r-2}
      \]
      and
      \begin{equation}\label{eq:genus}
        g_{n,r}=1+\frac12n^{r-2}\bigl((r-2)n-r\bigr).
      \end{equation}
\item Blowing up $q$ resolves the singularity in one step.  The exceptional
      curve is a reduced copy of $C_{n,r}$ with self-intersection
      \[
        C_{n,r}^2=-n^{r-2}.
      \]
\end{enumerate}
\end{proposition}

\begin{proof}
At $p$ exactly $k-r$ of the $k$ homogeneous coordinates are non-zero.  The
fiber of the power map is therefore a torsor under
$\mu_n^{k-r}/\mu_n$, giving $n^{k-r-1}$ points.

We first separate the cases with fewer than two vanishing arrangement coordinates.
If $r=0$, the power map is etale at every point of the fiber, so the completed
local ring is regular.  If $r=1$, choose local coordinates $(u,v)$ on the base
with the unique arrangement line given by $u=0$.  After extracting $n$-th roots
of the remaining units, each completed local branch has the form $t^n=u$, hence
is $\CC[[t,v]]$.  Thus $X_n$ is smooth over all points with $r\le1$.

Assume now $r\ge2$ and put $U=\{1,\ldots,k\}\setminus T_p$.
Let
\[
 W=\{(\lambda_1,\ldots,\lambda_k)\in\CC^k:
       \textstyle\sum_i\lambda_i\ell_i=0\}
\]
be the $(k-3)$-dimensional relation space.  The subspace $W_{T_p}$ of relations
supported on $T_p$ has dimension $r-2$, because the $r$ forms vanishing at $p$
span the two-dimensional space of linear forms vanishing at $p$.  Restriction
to the unit coordinates gives a linear map
\[
        \rho_U:W\longrightarrow\CC^U
\]
with kernel exactly $W_{T_p}$.  Evaluating a relation at $p$ shows that its image is
contained in the hyperplane
\[
 H_p=\Bigl\{(a_j)_{j\in U}:\sum_{j\in U}a_j\ell_j(p)=0\Bigr\}.
\]
Both $W/W_{T_p}$ and $H_p$ have dimension $k-r-1$, hence the induced map
$W/W_{T_p}\to H_p$ is an isomorphism.

Choose $j_0\in U$.  Since $\ell_{j_0}(p)\ne0$, projection
$H_p\to\CC^{U\setminus\{j_0\}}$ is an isomorphism.  Work in the projective
chart $z_{j_0}\ne0$ and normalize $z_{j_0}=1$.  Choose a basis
$\lambda^{(1)},\ldots,\lambda^{(k-3)}$ of $W$ whose first $r-2$ vectors span
$W_{T_p}$.  For the remaining $k-r-1$ equations
\[
        G_a(z)=\sum_i\lambda_i^{(a)}z_i^n \qquad (a=r-1,\ldots,k-3),
\]
the Jacobian with respect to the $k-r-1$ unit variables
$z_j$, $j\in U\setminus\{j_0\}$, at $q$ is
\[
 \bigl(\lambda_j^{(a)}\bigr)_{a,j}
 \operatorname{diag}\bigl(nq_j^{\,n-1}\bigr)_{j\in U\setminus\{j_0\}}.
\]
The first factor is invertible by the preceding identification with $H_p$, and
the diagonal factor is invertible because every $q_j$, $j\in U$, is nonzero.
The formal implicit-function theorem therefore eliminates precisely these
$k-r-1$ unit variables as formal power series in the vanishing coordinates
$z_i$, $i\in T_p$.  After subtracting their values at $q$, the eliminated unit
coordinates have zero constant term, hence lie in the ideal generated by the
vanishing coordinates.  Consequently the completed maximal ideal is carried
\emph{exactly} to the vertex ideal $(x_i:i\in T_p)$, not merely to an ideal with
the same radical.

Crucially, the first $r-2$ equations are supported on $T_p$ and are not altered by
this elimination.  Thus the vanishing coordinates themselves survive as the
cone coordinates, and we obtain the coordinate-preserving isomorphism
\eqref{eq:completed-local-cone}, with
$F_a=\sum_{i\in T}\lambda_i^{(a)}x_i^n$ for $a=1,\ldots,r-2$.
For $r=2$ there are no remaining equations, so the completed local ring is
regular.

For $r\ge3$, we also record explicitly that the projective curve $C_{n,r}$ is
integral.  Write $u_1,\ldots,u_r$ for the restrictions to the corresponding
$\PP^1$ of the $r$ arrangement forms through $p$.  Their zeroes are pairwise
distinct.  In $\CC(\PP^1)^\times/\CC(\PP^1)^{\times n}$ the classes
\[
        u_2/u_1,\ldots,u_r/u_1
\]
are independent: valuation at the zero of $u_j$ detects the exponent of
$u_j/u_1$.  Hence the generic Kummer algebra is a field, and the associated
function-field extension has degree $n^{r-1}$.  On the other hand, the restriction
$C_{n,r}\to\PP^1$ of the coordinatewise power map has degree
\[
        \deg\OO_{C_{n,r}}(n)=n\deg C_{n,r}=n^{r-1}.
\]
Because the map to $\PP^1$ is finite, every irreducible component would dominate
$\PP^1$; the generic fiber is the spectrum of the field above, so there is a
unique minimal point and its generic local ring is a field.  Since $C_{n,r}$ is a
complete intersection, hence Cohen--Macaulay and $S_1$,
Lemma~\ref{lem:s1-generic-field} shows that $C_{n,r}$ is integral.  It is smooth because the line $\PP^1$ meets each
coordinate hyperplane in a distinct point, the cover is etale
away from those points, and locally at each branch point it is obtained by
adjoining an $n$-th root of a parameter.  Since $C_{n,r}$ is a complete
intersection of type $(n,\ldots,n)$ with $r-2$ factors in $\PP^{r-1}$, its degree
is $n^{r-2}$ and adjunction gives
\[
 \omega_{C_{n,r}}\simeq
 \OO_{C_{n,r}}((r-2)n-r),
\]
which yields \eqref{eq:genus}.

The cone ring in \eqref{eq:completed-local-cone} is a standard graded normal
domain: it is a complete intersection, hence $S_2$, and it is regular away from
the vertex because $C_{n,r}$ is smooth; the vertex has codimension two.  Moreover
the coordinate-preserving statement above shows
that the maximal ideal of $\OO_{X_n,q}$ becomes the vertex ideal after
completion.  Lemma~\ref{lem:completed-cone-blowup} therefore applies.  It shows
that the ordinary blow-up of $q$ is smooth and that its exceptional fiber is a
reduced copy of $C_{n,r}$, with
\[
        \OO(-E)|_{C_{n,r}}\simeq\OO_{C_{n,r}}(1).
\]
Hence its self-intersection is
$-\deg\OO_{C_{n,r}}(1)=-n^{r-2}$.
\end{proof}

Let
\[
        \Sigma_n=\Sing(X_n)_{\mathrm{red}},
        \qquad
        \rho_n:Y_n:=\Bl_{\Sigma_n}X_n\longrightarrow X_n.
\]
By Proposition~\ref{prop:local}, $Y_n$ is smooth.  Write $E_{n,r}$ for the sum of
the exceptional curves above the $r$-fold points and
\[
        E_n=\sum_{r\ge3}E_{n,r}.
\]
The exceptional curves have self-intersection at most $-2$, so the resolution is
minimal.

\begin{corollary}\label{cor:intersection}
One has
\[
 A_n\cdot E_n=0,
 \qquad
 E_{n,r}^2=-t_r n^{k-3},
 \qquad
 E_{n,r}\cdot E_{n,r'}=0\quad(r\ne r'),
 \qquad
 E_n^2=-Tn^{k-3}.
\]
Moreover
\begin{equation}\label{eq:ideal-exceptional}
        \cI_{\Sigma_n/X_n}\OO_{Y_n}=\OO_{Y_n}(-E_n),
\end{equation}
with $E_n$ reduced.
\end{corollary}

\begin{proof}
There are $t_rn^{k-r-1}$ exceptional curves of type $(n,r)$, each of square
$-n^{r-2}$; hence
\[
 E_{n,r}^2=t_rn^{k-r-1}(-n^{r-2})=-t_rn^{k-3}.
\]  The curves are
pairwise disjoint, and $A_n$ is pulled back from $X_n$, hence has zero
intersection with them.  Lemma~\ref{lem:completed-cone-blowup}, applied through the coordinate-preserving
local model of Proposition~\ref{prop:local}, shows that the maximal ideal pulls
back with order exactly one along the reduced exceptional curve; this is
precisely \eqref{eq:ideal-exceptional}.
\end{proof}

\subsection{Canonical class, Euler number, and signature}

Set
\[
        m_n=(k-3)n-k,
        \qquad
        c_{n,r}=(r-2)(n-1)-1.
\]

\begin{proposition}\label{prop:chern}
The surface $Y_n$ satisfies
\begin{equation}\label{eq:K-general}
        K_{Y_n}=m_nA_n-\sum_{r\ge3}c_{n,r}E_{n,r},
\end{equation}
and therefore
\begin{equation}\label{eq:K2-direct}
 \frac{K_{Y_n}^2}{n^{k-3}}
 =m_n^2-\sum_{r\ge3}c_{n,r}^2t_r.
\end{equation}
Furthermore
\begin{equation}\label{eq:c2-strata}
\begin{aligned}
 c_2(Y_n)=e(Y_n)
 &=n^{k-1}e(U)+n^{k-2}(2k-f_1)+t_2n^{k-3}\\
 &\quad+\sum_{r\ge3}t_r n^{k-r-1}(2-2g_{n,r}),
\end{aligned}
\end{equation}
where
\begin{equation}\label{eq:eU}
        e(U)=3-2k+f_1-f_0.
\end{equation}
Equivalently,
\begin{align}
 \frac{c_2(Y_n)}{n^{k-3}}
 &=n^2(3-2k+f_1-f_0)
   +n(2k-2f_1+2f_0)+f_1-t_2,\label{eq:c2-closed}\\
 \frac{K_{Y_n}^2}{n^{k-3}}
 &=n^2(9-5k+3f_1-4f_0)
   +2n(2k-2f_1+2f_0)\\
 &\qquad +k+f_1-f_0+t_2.\label{eq:K2-closed}
\end{align}
\end{proposition}

\begin{proof}
Adjunction for the complete intersection $X_n\subset\PP^{k-1}$ gives
\[
        \omega_{X_n}\simeq\OO_{X_n}((k-3)n-k)=\OO_{X_n}(m_n).
\]
In particular $X_n$ is Gorenstein and $K_{X_n}$ is Cartier.  Since
$\rho_n:Y_n\to X_n$ is an isomorphism away from the isolated singular points,
$K_{Y_n}-\rho_n^*K_{X_n}$ is supported on the exceptional curves.  Their pairwise
disjointness therefore allows the coefficients to be determined one curve at a
time.  For an exceptional curve $C=C_{n,r}$, adjunction on $Y_n$ gives
\[
 K_{Y_n}\cdot C=2g_{n,r}-2-C^2
 =n^{r-2}\bigl((r-2)n-r+1\bigr).
\]
Since $C^2=-n^{r-2}$ and $A_n\cdot C=0$, its discrepancy is
\[
 -\bigl((r-2)n-r+1\bigr)
 =-\bigl((r-2)(n-1)-1\bigr)=-c_{n,r},
\]
which proves \eqref{eq:K-general}; \eqref{eq:K2-direct} follows from
Corollary~\ref{cor:intersection}.

For the Euler number, stratify $\PP^2$ into the complement $U$, the smooth
parts of the arrangement lines, and the multiple points.  The total Euler
characteristic of the open line strata is
\[
        2k-f_1.
\]
The fiber cardinalities of $\pi_n$ over these strata are respectively
$n^{k-1}$, $n^{k-2}$, and $n^{k-r-1}$ over an $r$-fold point.  A double point is
smooth on $X_n$, while for $r\ge3$ the resolution replaces each singular point
by a curve of Euler characteristic $2-2g_{n,r}$.  This gives
\eqref{eq:c2-strata}.  Formula \eqref{eq:eU} follows from
$e(\PP^2)=3$.

Using \eqref{eq:genus}, the last sum in \eqref{eq:c2-strata} simplifies to
\[
 -n^{k-3}\sum_{r\ge3}t_r\bigl((r-2)n-r\bigr),
\]
which gives \eqref{eq:c2-closed}.  Expanding \eqref{eq:K2-direct} and using the
incidence identity \eqref{eq:pair-identity} gives \eqref{eq:K2-closed}.
These formulas are derived directly in the projective normalization used in this
paper.  The Hirzebruch--Kummer construction itself is classical
\cite{Hirzebruch,Pokora2019}; no identification with the alternative
``blow up first, then cover'' resolution is used in the derivation above.
\end{proof}

\begin{corollary}[Signature formula]\label{cor:signature-formula}
Define
\begin{equation}\label{eq:QR}
 Q_{\cA}=3-k+\sum_{r\ge2}(r-2)t_r,
 \qquad
 R_{\cA}=k+3t_2-f_1-f_0.
\end{equation}
Then
\begin{equation}\label{eq:signature-formula}
 \boxed{
 \sig(Y_n)=\frac{n^{k-3}}3\bigl(Q_{\cA}n^2+R_{\cA}\bigr).}
\end{equation}
In particular the coefficient of $n$ vanishes identically.
\end{corollary}

\begin{proof}
By Noether's formula,
\[
        \sig(Y_n)=\frac{K_{Y_n}^2-2c_2(Y_n)}3.
\]
Subtract twice \eqref{eq:c2-closed} from \eqref{eq:K2-closed}.  The coefficient
of $n$ cancels, and the remaining coefficients are exactly \eqref{eq:QR}.
\end{proof}

\begin{remark}[Consistency with the classical Kummer-cover formulas]\label{rem:classical-consistency}
The numerical formulas above were derived directly from the projective model and
its local resolution.  As an external consistency check, the ratio of the leading
coefficients in \eqref{eq:K2-closed} and \eqref{eq:c2-closed} is
\[
 \frac{9-5k+3f_1-4f_0}{3-2k+f_1-f_0},
\]
which is the classical Hirzebruch characteristic number of a line arrangement;
see \cite{Hirzebruch,BarthelHirzebruchHoefer,Pokora2019}.  Likewise, for the Hesse,
dual-Hesse, and complete-quadrangle data treated below, our formulas give
$K^2=3c_2$ at the classical exponents $n=3,5,5$, respectively.  These agreements
provide independent numerical checks only; no identification with an alternative
``blow up first, then cover'' model is used in any proof.
\end{remark}

\section{Very ampleness from the Rees algebra}\label{sec:rees}

The next lemma is the standard very-ampleness consequence of the Rees-algebra
realization of a blow-up and the Segre embedding.  We include the proof to fix the
precise line bundle used later.

\begin{lemma}[Very ampleness from a generated ideal]\label{lem:rees-segre}
Let $X$ be a projective scheme with $\OO_X(1)$ very ample, let
$\cI\subset\OO_X$ be a coherent ideal sheaf, and let
\[
        \rho:Y=\Bl_{\cI}X\to X,
        \qquad
        A=\rho^*\OO_X(1),
\]
with exceptional Cartier divisor $E$ defined by
$\cI\OO_Y=\OO_Y(-E)$.  If $\cI(\nu)$ is globally generated for some
$\nu\ge1$, then
\[
        (\nu+1)A-E
\]
is very ample on $Y$.
\end{lemma}

\begin{proof}
Choose a finite-dimensional vector space $V$ and a surjection
\[
        V\otimes\OO_X\twoheadrightarrow\cI(\nu).
\]
Let
\[
 \mathcal R=\bigoplus_{m\ge0}\cI^m,
 \qquad
 \mathcal R^{(\nu)}=\bigoplus_{m\ge0}\cI^m(\nu m).
\]
Twisting the $m$-th graded piece by the $m$-th power of an invertible sheaf does
not change relative Proj, and under the resulting identification
\[
 \OO_{\operatorname{Proj}\mathcal R^{(\nu)}}(1)
 \simeq \OO_Y(-E)\otimes\rho^*\OO_X(\nu)
 =\OO_Y(\nu A-E);
\]
see \cite[II, Lemma~7.9]{Hartshorne}.  Since the Rees algebra is generated in
degree one, the chosen sections induce a surjection of graded algebras
\[
 \operatorname{Sym}_{\OO_X}(V\otimes\OO_X)
 \twoheadrightarrow\mathcal R^{(\nu)}.
\]
Taking relative Proj gives a closed immersion
\[
        Y\hookrightarrow X\times\PP(V^\vee)
\]
whose pullback of $\OO_{\PP(V^\vee)}(1)$ is $\OO_Y(\nu A-E)$; compare
\cite[II,~3.6.2]{EGAII}.  The Segre line bundle
$\OO_X(1)\boxtimes\OO_{\PP(V^\vee)}(1)$ is very ample on the product, and its
restriction to $Y$ is $(\nu+1)A-E$, proving very ampleness.
\end{proof}

We now record an arrangement-theoretic way to produce the required generators.
For every point $p$ of multiplicity at least three, let
\[
        T_p=\{i:L_i\ni p\}\subset\{1,\ldots,k\}.
\]
Set
\begin{equation}\label{eq:arrangement-ideal}
        J_{\cA}=\bigcap_{r(p)\ge3}(z_i:i\in T_p)
        \subset\CC[z_1,\ldots,z_k].
\end{equation}

\begin{lemma}\label{lem:singular-ideal-general}
For every $n\ge2$,
\[
        J_{\cA}\OO_{X_n}=\cI_{\Sigma_n/X_n}.
\]
\end{lemma}

\begin{proof}
Let $q\in\Sigma_n$ lie over a multiple point $p$.  Exactly the coordinates
$z_i$ with $i\in T_p$ vanish at $q$.  For any other multiple point $p'\ne p$,
some line through $p'$ does not pass through $p$, so the corresponding
coordinate is a unit at $q$; hence the factor indexed by $p'$ in
\eqref{eq:arrangement-ideal} becomes the unit ideal.  Therefore
\[
        J_{\cA}\OO_{X_n,q}=(z_i:i\in T_p)\OO_{X_n,q}.
\]
To identify this ideal scheme-theoretically, pass to the maximal-ideal-adic
completion.  Under the coordinate-preserving completed isomorphism
\eqref{eq:completed-local-cone} in Proposition~\ref{prop:local}, the images of
the ambient coordinates $z_i$, $i\in T_p$, are exactly the cone coordinates
$x_i$.  They therefore generate the completed maximal ideal:
\[
 (z_i:i\in T_p)\widehat{\OO}_{X_n,q}
     =\mathfrak m_q\widehat{\OO}_{X_n,q}.
\]
Since $\OO_{X_n,q}\to\widehat{\OO}_{X_n,q}$ is faithfully flat, equality after
completion descends to
\[
        (z_i:i\in T_p)\OO_{X_n,q}=\mathfrak m_q.
\]
At a point outside $\Sigma_n$, every factor in
\eqref{eq:arrangement-ideal} contains a unit locally, so
$J_{\cA}\OO_{X_n}$ is the unit ideal there.  This proves the asserted equality
of ideal sheaves, not merely equality of their supports.
\end{proof}

We will repeatedly use the following elementary translation.  If $q\in X_n$ and
$x\in\PP^2$ is the point corresponding to $\pi_n(q)\in\Lambda$, then
\begin{equation}\label{eq:zero-coordinate-translation}
        z_i(q)=0\quad\Longleftrightarrow\quad \ell_i(x)=0.
\end{equation}
Indeed $\pi_n(q)=[z_1(q)^n:\cdots:z_k(q)^n]$ and, under
$\Lambda\simeq\PP^2$, this point is
$[\ell_1(x):\cdots:\ell_k(x)]$.  Thus a square-free monomial
$\prod_{i\in S}z_i$ vanishes at $q$ exactly when the product
$\prod_{i\in S}\ell_i$ vanishes at $x$.

\begin{definition}[A generation datum]\label{def:generation-datum}
A degree-$\nu$ generation datum for $\cA$ is a finite collection
$\mathcal S$ of subsets $S\subset\{1,\ldots,k\}$ of cardinality $\nu$ such that:
\begin{enumerate}[label=\textup{(G\arabic*)}]
\item each $S$ meets every $T_p$;
\item for every multiple point $p$ and every $i\in T_p$, some
      $S\in\mathcal S$ satisfies $S\cap T_p=\{i\}$;
\item the common zero locus on $\PP^2$ of the monomials
      $\prod_{i\in S}\ell_i$, $S\in\mathcal S$, is precisely the set of
      points of multiplicity at least three.
\end{enumerate}
\end{definition}

\begin{proposition}\label{prop:generation-datum}
If $\cA$ has a degree-$\nu$ generation datum, then for every $n\ge2$ the
monomials
\[
        M_S=\prod_{i\in S}z_i,\qquad S\in\mathcal S,
\]
globally generate $\cI_{\Sigma_n/X_n}(\nu)$.  Consequently
\begin{equation}\label{eq:H-general}
        H_n=(\nu+1)A_n-E_n
\end{equation}
is very ample on $Y_n$.
\end{proposition}

\begin{proof}
By (G1), each $M_S$ lies in $J_{\cA}$.  At $q$ over a multiple point $p$,
$M_S$ localizes to a unit times
$\prod_{i\in S\cap T_p}z_i$.  Condition (G2) therefore gives, in the stalk at $q$, elements whose images
generate the maximal ideal $(z_i:i\in T_p)$.  By Nakayama's lemma the chosen
sections generate the stalk of $\cI_{\Sigma_n/X_n}(\nu)$ at every point of
$\Sigma_n$.  By \eqref{eq:zero-coordinate-translation}, condition (G3) says
that at every point outside $\Sigma_n$ at least one chosen section is nonzero,
so the stalk there is also generated.  Lemma~\ref{lem:singular-ideal-general}
then identifies the ideal sheaf, and Lemma~\ref{lem:rees-segre} gives very
ampleness.
\end{proof}

\section{The Hesse family}\label{sec:hesse}

Let $\cH$ be the classical Hesse arrangement of twelve lines.  It consists of
the four singular triangles of the Hesse pencil.  Its nine base points are
fourfold points and there are twelve additional double points:
\begin{equation}\label{eq:hesse-incidence}
        k=12,\qquad t_2=12,\qquad t_4=9.
\end{equation}
Let the four triangles have line sets
\[
        \Pi_1,\Pi_2,\Pi_3,\Pi_4,
        \qquad |\Pi_j|=3,
\]
and put
\[
        g_j=\prod_{i\in\Pi_j}z_i\qquad(j=1,\ldots,4).
\]

\begin{proposition}\label{prop:hesse-generation}
For every $n\ge2$, the four cubic monomials $g_1,\ldots,g_4$ globally generate
$\cI_{\Sigma_n/X_n}(3)$.  Consequently
\[
        H_n=4A_n-E_n
\]
is very ample on $Y_n$.
\end{proposition}

\begin{proof}
At each of the nine fourfold points, exactly one line from each of the four
triangles passes through the point.  Hence, at a point $q\in\Sigma_n$ above such
a base point, each $g_j$ is a unit times one of the four vanishing Kummer
coordinates.  The four monomials therefore generate the maximal ideal
$\mathfrak m_q$.

The common zero locus in $\PP^2$ of the four products
$\prod_{i\in\Pi_j}\ell_i$ is the base locus of the Hesse pencil, namely the nine
fourfold points.  Thus the four monomials have no common zero on $X_n$ outside
$\Sigma_n$.  Proposition~\ref{prop:generation-datum} applies with $\nu=3$.
\end{proof}

\begin{remark}[An explicit ambient projective space]\label{rem:hesse-ambient}
The four cubic generators give a closed immersion
\[
   Y_n\hookrightarrow X_n\times\PP^3.
\]
For the Hesse arrangement $X_n\subset\PP^{11}$, and composition with the Segre
embedding gives an explicit closed immersion $Y_n\hookrightarrow\PP^{47}$.  The
pullback of $\OO_{\PP^{11}}(1)\boxtimes\OO_{\PP^3}(1)$ is
$4A_n-E_n=H_n$.  Thus this closed immersion is defined by a finite-dimensional
subspace of $H^0(Y_n,H_n)$ and serves as an explicit witness that $H_n$ is very
ample.  In particular the complete linear system $|H_n|$ also embeds $Y_n$; this is
the embedding we may take in the ESW statement.  No formula for the dimension of
the complete linear system is needed for the argument.
\end{remark}

\begin{proposition}[Independent ampleness check and sharpness]\label{prop:hesse-sharp}
For every $n\ge2$, the divisor $4A_n-E_n$ on the Hesse Kummer resolution is
ample.  The divisor $3A_n-E_n$ is nef but not ample.  Thus the coefficient $4$
is the first integral coefficient in the ray $\alpha A_n-E_n$ that can be ample.
This proposition is independent of the Rees-algebra proof; the latter remains
essential above because it proves \emph{very} ampleness.
\end{proposition}

\begin{proof}
Let $B$ be the set of the nine fourfold points of the Hesse arrangement.  Let
$C\subset\PP^2$ be an irreducible curve of degree $d$ that is not an arrangement
line, and let $\mu_p=\operatorname{mult}_p(C)$ for $p\in B$.  Summing B\'ezout
against the twelve arrangement lines gives
\[
  4\sum_{p\in B}\mu_p\le 12d,
  \qquad\text{hence}\qquad
  \sum_{p\in B}\mu_p\le3d,
\]
because every point of $B$ lies on four arrangement lines.

Let $\widetilde Z$ be a component of the strict transform on $Y_n$ of the reduced
inverse image of $C$.  For the full inverse image, $\pi_n^*C\equiv ndA_n$.  At a
point $q$ above $p\in B$, the pullback of a local equation of $C$ has order
exactly $n\mu_p$ at the cone vertex.  Indeed, choose local coordinates $u,v$ at
$p$.  In the cone model of Proposition~\ref{prop:local}, the pullbacks of $u$ and
$v$ have initial homogeneous parts $U_n,V_n$ of degree $n$ in the cone
coordinates.  If $f_{\mu_p}(u,v)$ is the nonzero degree-$\mu_p$ initial form of
a local equation of $C$, then the first possible homogeneous term of its pullback
is $f_{\mu_p}(U_n,V_n)$, of degree $n\mu_p$.  This term is nonzero in the cone
ring: on the exceptional curve $C_{n,4}$ it is the pullback of the nonzero
homogeneous polynomial $f_{\mu_p}$ along the finite surjective map
$C_{n,4}\to\PP^1$ of tangent directions.  Hence no cancellation raises the
order.  There are $n^7$ points of $X_n$ above each fourfold point, and each
corresponding exceptional curve has square $-n^2$.  Therefore, for
$H_{\alpha}=\alpha A_n-E_n$, the full strict transform satisfies
\[
 H_{\alpha}\cdot \widetilde{\pi_n^{-1}(C)}
   =n^{10}\left(\alpha d-\sum_{p\in B}\mu_p\right).
\]
If the inverse image is reducible, its irreducible components correspond to the
height-one primes of the normal Kummer cover lying over the generic point of
$C$.  The generic Kummer extension is Galois, and its Galois (diagonal deck)
group acts transitively on the primes above a fixed prime of the base.  Thus it
acts transitively on these components.  Since the action preserves $A_n$ and
$E_n$, all components have the same intersection with $H_{\alpha}$ and hence the
same sign.  Therefore $H_4:=4A_n-E_n$ has positive intersection with
every component lying over such a curve, while $H_3:=3A_n-E_n$ has nonnegative
intersection.

For an arrangement line, the same intersection calculation (or direct use of
its three fourfold points) gives
\[
       H_{\alpha}\cdot\widetilde{\pi_n^{-1}(L_i)}
       =(\alpha-3)n^9.
\]
Finally, for every exceptional curve $E_q\simeq C_{n,4}$ one has
$H_{\alpha}\cdot E_q=n^2>0$.  Every irreducible curve on $Y_n$ is either
exceptional or maps to a curve of $\PP^2$, so these cases exhaust the curves.
Moreover
\[
      (4A_n-E_n)^2=7n^9>0,
      \qquad
      (3A_n-E_n)^2=0.
\]
Nakai--Moishezon therefore gives ampleness of $4A_n-E_n$, while the intersection
inequalities show that $3A_n-E_n$ is nef; its square zero prevents ampleness.
\end{proof}

\begin{remark}[Why the cubic generation is critical]\label{rem:hesse-sharp-degree}
The preceding proposition shows that the Hesse construction sits on the first
integral ample level above the nef boundary.  The degree-$3$ generation in
Proposition~\ref{prop:hesse-generation} is correspondingly structural, not a
cosmetic optimization.  If one only knew generation in degree $4$, the standard
Rees-algebra/Segre construction would give $5A_3-E_3$, and
\[
       (5A_3-E_3)^2=(25-9)3^9=16\cdot3^9=\sig(Y_3).
\]
The strict numerical obstruction would then give no counterexample at the
headline exponent $n=3$ (although the Hesse family would enter the strict
non-existence range again for $n\ge4$).
\end{remark}

\begin{remark}[The critical degree $n=3$]\label{rem:hesse-n3-degree}
For the headline member $n=3$, the homogeneous ideal of
$X_3\subset\PP^{11}$ already contains the cubic relation equations
$\sum_i\lambda_i z_i^3$.  This does not create a relation among the four
sections $g_1,\ldots,g_4$: the degree-three defining relations are linear
combinations of the pure cubes $z_i^3$, whereas each $g_j$ is a square-free
monomial in three distinct variables.  Hence the $g_j$ remain linearly
independent in $H^0(X_3,\OO_{X_3}(3))$.  Likewise, although the four products
$\prod_{i\in\Pi_j}\ell_i$ downstairs belong to the two-dimensional Hesse pencil,
that linear dependence does not lift to the Kummer monomials $g_j$; only their
$n$-th powers are related to the corresponding products under the power map.
The local generation argument above is therefore unchanged at $n=3$.
\end{remark}

For \eqref{eq:hesse-incidence}, one has
\[
 f_0=21,\qquad f_1=60,\qquad T=9,
\]
and therefore
\[
 Q_{\cH}=9,\qquad R_{\cH}=-33.
\]

\begin{theorem}[Hesse family]\label{thm:hesse-family}
For every integer $n\ge3$, the smooth polarized surface
\[
        (Y_n,H_n),\qquad H_n=4A_n-E_n,
\]
associated with the exponent-$n$ Kummer cover of the Hesse arrangement carries
no Ulrich bundle of any rank.  More precisely,
\begin{equation}\label{eq:hesse-numerics}
        H_n^2=7n^9,
        \qquad
        \sig(Y_n)=(3n^2-11)n^9,
\end{equation}
and hence
\begin{equation}\label{eq:hesse-ratio}
        \frac{H_n^2}{\sig(Y_n)}=\frac7{3n^2-11}\longrightarrow0.
\end{equation}
Equivalently, $\uc(Y_n,H_n)=\infty$ for every $n\ge3$.
\end{theorem}

\begin{proof}
Proposition~\ref{prop:hesse-generation} gives very ampleness.  Since
$A_n^2=n^9$, $E_n^2=-9n^9$, and $A_n\cdot E_n=0$, one gets
\[
        H_n^2=(4A_n-E_n)^2=7n^9.
\]
For the Hesse incidence data, Corollary~\ref{cor:signature-formula} with
$Q=9$ and $R=-33$ gives
\[
        \sig(Y_n)=(3n^2-11)n^9.
\]
Thus $H_n^2<\sig(Y_n)$ exactly when $7<3n^2-11$, i.e. for $n\ge3$.
Proposition~\ref{prop:obstruction} then excludes Ulrich bundles of every rank.
\end{proof}

The Hesse members are pairwise non-isomorphic.  Indeed
\begin{equation}\label{eq:hesse-K2-growth}
 K_{Y_n}^2=n^9(45n^2-108n+63)
           =9n^9(n-1)(5n-7),
\end{equation}
which is strictly increasing for $n\ge3$.  Thus the Hesse family alone already
contains infinitely many pairwise non-isomorphic smooth surfaces.

\subsection{\texorpdfstring{The ball-quotient member $n=3$}{The ball-quotient member n=3}}

At $n=3$, Proposition~\ref{prop:chern} gives
\[
        K_{Y_3}=15A_3-3E_3=3(5A_3-E_3),
\]
and
\[
        K_{Y_3}^2=144\cdot3^9,
        \qquad
        c_2(Y_3)=48\cdot3^9.
\]
Since $H_3=4A_3-E_3$ is ample and $A_3$ is nef,
$5A_3-E_3=H_3+A_3$ is ample.  Thus $K_{Y_3}$ is ample and
$K_{Y_3}^2=3c_2(Y_3)$; by the equality case of Miyaoka--Yau and its uniformization theorem
\cite{Miyaoka1977,Yau1977,GrebKebekusPeternellTaji}, $Y_3$ is a compact ball quotient.  This direct
argument is all that is used below; no identification with a particular
classical resolution is needed.  Moreover
\[
        \chi(\OO_{Y_3})=16\cdot3^9,
        \qquad
        \sig(Y_3)=16\cdot3^9=314928,
\]
whereas
\[
        H_3^2=7\cdot3^9=137781.
\]
If
\[
        L=K_{Y_3}/3=5A_3-E_3,
\]
then
\[
        H_3=L-A_3.
\]
Thus this first member realizes Proposition~\ref{prop:ball-displacement} in the
simplest possible form.

\section{The general non-existence mechanism}\label{sec:mechanism}

We now combine the signature formula and the Rees-algebra/Segre construction.

\begin{theorem}[Kummer-family mechanism]\label{thm:kummer-family}
Let $\cA$ be a line arrangement with a degree-$\nu$ generation datum, and set
\[
 a=\nu+1,
 \qquad
 T=\sum_{r\ge3}t_r,
\]
with $Q_{\cA},R_{\cA}$ as in \eqref{eq:QR}.  Then for every $n\ge2$ the divisor
\[
        H_n=aA_n-E_n
\]
is very ample and
\begin{align}
 H_n^2&=(a^2-T)n^{k-3},\label{eq:H2-family}\\
 \sig(Y_n)&=\frac{n^{k-3}}3(Q_{\cA}n^2+R_{\cA}).\label{eq:sig-family}
\end{align}
If
\begin{equation}\label{eq:family-threshold}
        Q_{\cA}n^2+R_{\cA}>3(a^2-T),
\end{equation}
then $(Y_n,H_n)$ carries no Ulrich bundle of any rank.
If $Q_{\cA}>0$, condition \eqref{eq:family-threshold} holds for all sufficiently
large $n$ and
\begin{equation}\label{eq:ratio-family}
 \frac{H_n^2}{\sig(Y_n)}
 =\frac{3(a^2-T)}{Q_{\cA}n^2+R_{\cA}}
 \xrightarrow[n\to\infty]{}0.
\end{equation}
\end{theorem}

\begin{proof}
Very ampleness follows from Proposition~\ref{prop:generation-datum}; in particular
$H_n^2>0$.  Thus the formula below also shows automatically that
$a^2>T$.  By Corollary~\ref{cor:intersection},
\[
 H_n^2=a^2A_n^2+E_n^2
       =(a^2-T)n^{k-3},
\]
while \eqref{eq:sig-family} is Corollary~\ref{cor:signature-formula}.
Condition \eqref{eq:family-threshold} is exactly $H_n^2<\sig(Y_n)$, so
Proposition~\ref{prop:obstruction} applies.  The final assertion is immediate
when $Q_{\cA}>0$.
\end{proof}

\begin{remark}\label{rem:not-ball-essential}
The ball-quotient property plays no role in Theorem~\ref{thm:kummer-family}.
Ball quotients are geometrically distinguished because they saturate the
Bogomolov--Miyaoka--Yau inequality \cite{Miyaoka1977,Yau1977} and therefore make
$\sig=K^2-8\chi$ especially transparent, but the infinite-family phenomenon is
driven by the quadratic growth in \eqref{eq:sig-family} and fixed-degree
Rees generation.
\end{remark}

\section{Two companion families}\label{sec:companions}

\subsection{The dual Hesse family}

The dual Hesse arrangement has
\[
        k=9,\qquad t_3=12,
\]
and its twelve triples form the affine plane $AG(2,3)$ on the nine lines.
Let
\[
        J=\bigcap_{T\in AG(2,3)}(z_i:i\in T).
\]

\begin{lemma}\label{lem:dual-hesse-degree5}
The ideal $J$ is generated by square-free monomials of degree $5$.  Consequently,
for every Kummer exponent $n\ge2$,
\[
        \cI_{\Sigma_n/X_n}(5)
\]
is globally generated and
\[
        H_n=6A_n-E_n
\]
is very ample.
\end{lemma}

\begin{proof}
A monomial belongs to
$J=\bigcap_T(z_i:i\in T)$ exactly when its support meets every affine line
$T\subset AG(2,3)$.  Since membership depends only on the support, the minimal
monomial generators of $J$ are square-free and their supports are precisely the
inclusion-minimal vertex covers of the twelve-line hypergraph.  Complements of
minimal vertex covers are maximal caps, where a cap is a subset containing no
affine line.

We first show that every maximal cap has four points.  Any cap of cardinality at
most two extends to a three-point cap.  A three-point cap is non-collinear, so
after an affine change of coordinates it is
\[
        (0,0),\qquad (1,0),\qquad (0,1).
\]
The third points on the three affine lines determined by these pairs are
$(2,0)$, $(0,2)$ and $(2,2)$, so these are forbidden.  The remaining points are
$(1,1)$, $(1,2)$ and $(2,1)$; each can be added individually.  Any two of these
three candidates, together with respectively one of $(1,0)$, $(0,1)$ and
$(0,0)$, form an affine line.  Hence no cap has more than four points and every
cap of size at most three extends to one of size four.  Thus every maximal cap
has cardinality four, every minimal vertex cover has cardinality $9-4=5$, and
$J$ is generated by square-free monomials of degree five.

We next verify the local generation condition (G2).  The affine group
$\mathrm{AGL}(2,3)$ is transitive on incident point--line flags: translate the
chosen point to the origin and send the direction of the chosen affine line to
the horizontal direction.  It is therefore enough to treat one flag.  Take
\[
 T=\{(0,0),(1,0),(2,0)\},\qquad i=(2,0).
\]
Then
\[
 C=\{(0,0),(1,0),(0,1),(1,2)\}
\]
is a four-point cap.  Its complement $S$ is a minimal vertex cover of cardinality
five and satisfies $S\cap T=\{i\}$.  Flag transitivity gives the same conclusion
for every pair $(T,i)$ with $i\in T$.  Hence at a singular point over $T$ the
corresponding degree-five cover monomials localize to a unit times each of the
three coordinates $z_i$, $i\in T$, and generate the maximal ideal.

Finally we verify (G3).  A point of $\PP^2$ outside the twelve triple points lies
on at most one arrangement line.  If it lies on none, every cover monomial is
nonzero there.  If it lies on the single line indexed by $i$, choose a four-cap
containing $i$ (for example the cap $C$ above contains $(0,0)$, and point
transitivity gives such a cap through every $i$).  The complementary degree-five
cover $S$ avoids $i$, so its product of arrangement equations is nonzero at the
point.  By \eqref{eq:zero-coordinate-translation}, the corresponding monomial
$M_S$ is nonzero at every point of $X_n$ above it.  Thus the degree-five cover
monomials have no common zero away from $\Sigma_n$.

Conditions (G1)--(G3) are therefore satisfied in degree five, so
Proposition~\ref{prop:generation-datum} gives global generation of
$\cI_{\Sigma_n/X_n}(5)$, and Lemma~\ref{lem:rees-segre} gives the very ample
class $6A_n-E_n$.
\end{proof}

Here
\[
 f_0=12,\quad f_1=36,\quad T=12,
 \quad Q=6,\quad R=-39.
\]
Therefore
\begin{equation}\label{eq:dual-family}
 H_n^2=24n^6,
 \qquad
 \sig(Y_n)=(2n^2-13)n^6.
\end{equation}

\begin{corollary}[Dual Hesse family]\label{cor:dual-family}
For every $n\ge5$, the polarized surface $(Y_n,6A_n-E_n)$ carries no Ulrich
bundle of any rank, and
\[
        \frac{H_n^2}{\sig(Y_n)}=\frac{24}{2n^2-13}\longrightarrow0.
\]
At $n=5$ one has
\[
 K_{Y_5}=21A_5-3E_5=3(7A_5-E_5),
\]
$K_{Y_5}^2=3c_2(Y_5)$.  Since $6A_5-E_5$ is ample and $A_5$ is nef,
$7A_5-E_5$ is ample; hence $K_{Y_5}$ is ample and the equality case of
Miyaoka--Yau and its uniformization theorem
\cite{Miyaoka1977,Yau1977,GrebKebekusPeternellTaji} shows that $Y_5$ is a compact ball
quotient.  No identification with a particular classical dual-Hesse resolution
is used in the argument.
\end{corollary}

\subsection{The complete quadrangle}

Let $p_1,p_2,p_3,p_4\in\PP^2$ be four points in general position and let
$L_{ij}$, $1\le i<j\le4$, be the six joining lines.  The resulting complete
quadrangle has
\[
        k=6,\qquad t_2=3,\qquad t_3=4.
\]
The four triple points are $p_1,\ldots,p_4$.

Consider the three quadratic monomials corresponding to the three perfect
matchings of $K_4$:
\[
 g_1=z_{12}z_{34},\qquad
 g_2=z_{13}z_{24},\qquad
 g_3=z_{14}z_{23}.
\]

\begin{lemma}\label{lem:quadrangle-generation}
For every $n\ge2$, the three quadrics $g_1,g_2,g_3$ globally generate
$\cI_{\Sigma_n/X_n}(2)$.  Hence
\[
        H_n=3A_n-E_n
\]
is very ample.
\end{lemma}

\begin{proof}
At a vertex $p_i$, exactly three arrangement lines vanish.  Each perfect matching
contains exactly one of those three lines, so $g_1,g_2,g_3$ localize to units
times the three local Kummer coordinates and generate the maximal ideal.
Their common zero locus on $\PP^2$ is exactly the four vertices.  The three
diagonal points are precisely the three double points of the arrangement; by
Proposition~\ref{prop:local} they are smooth on $X_n$ and do not belong to
$\Sigma_n$.  At each such double point only one matching monomial vanishes, while
on a generic point of an arrangement line only the monomial containing that line
vanishes.  Thus
Proposition~\ref{prop:generation-datum} applies with $\nu=2$.
\end{proof}

For this arrangement
\[
 f_0=7,\quad f_1=18,\quad T=4,
 \quad Q=1,\quad R=-10,
\]
and hence
\begin{equation}\label{eq:quadr-family}
 H_n^2=5n^3,
 \qquad
 \sig(Y_n)=\frac{n^3}{3}(n^2-10).
\end{equation}

\begin{corollary}[Complete quadrangle family]\label{cor:quadr-family}
For every $n\ge6$, the polarized surface $(Y_n,3A_n-E_n)$ carries no Ulrich
bundle of any rank, and
\[
        \frac{H_n^2}{\sig(Y_n)}=\frac{15}{n^2-10}\longrightarrow0.
\]
For $n=5$ one has the equality
\[
        H_5^2=\sig(Y_5)=625.
\]
Moreover
\[
        K_{Y_5}=9A_5-3E_5=3H_5.
\]
Hence $K_{Y_5}$ is ample and $K_{Y_5}^2=9H_5^2=9\sig(Y_5)$.  Since
$\sig=(K^2-2c_2)/3$, this equality is equivalent to $K_{Y_5}^2=3c_2(Y_5)$;
the equality case of Miyaoka--Yau and its uniformization theorem
\cite{Miyaoka1977,Yau1977,GrebKebekusPeternellTaji} therefore
shows that $Y_5$ is a compact ball quotient.
Thus the exponent-$5$ complete-quadrangle member realizes, in higher Picard
rank, the numerical boundary configuration $K=3H$ and $H^2=\sig$ appearing in
Beauville's discussion.
\end{corollary}

\section{Comparison of the three families}

For convenience we collect the normalized formulas.  In every row $A_n^2=n^{k-3}$.

\begin{center}
\begin{tabular}{>{\raggedright\arraybackslash}p{4.1cm}cccc}
\toprule
Arrangement & very ample $H_n$ & $H_n^2/A_n^2$ & $\sig(Y_n)/A_n^2$ & no-go range\\
\midrule
Complete quadrangle & $3A_n-E_n$ & $5$ & $(n^2-10)/3$ & $n\ge6$\\
Dual Hesse & $6A_n-E_n$ & $24$ & $2n^2-13$ & $n\ge5$\\
Hesse & $4A_n-E_n$ & $7$ & $3n^2-11$ & $n\ge3$\\
\bottomrule
\end{tabular}
\end{center}

The same comparison is especially transparent in terms of the Chern slope
$K^2/\chi(\OO)$.  Directly from Proposition~\ref{prop:chern} one obtains
\[
\begin{array}{c|c|c}
\text{Arrangement}
& K_{Y_n}^2/\chi(\OO_{Y_n})
& \displaystyle\lim_{n\to\infty}K_{Y_n}^2/\chi(\OO_{Y_n})\\
\hline
\text{Complete quadrangle}
& \displaystyle\frac{60(n-2)^2}{7n^2-30n+35}
& \displaystyle\frac{60}{7}\approx 8.5714\\[6pt]
\text{Dual Hesse}
& \displaystyle\frac{12(8n^2-20n+11)}{11n^2-30n+23}
& \displaystyle\frac{96}{11}\approx 8.7273\\[6pt]
\text{Hesse}
& \displaystyle\frac{36(n-1)(5n-7)}{21n^2-54n+37}
& \displaystyle\frac{60}{7}\approx 8.5714
\end{array}
\]
Thus, throughout the ranges relevant to the counterexamples, the surfaces lie in
the Beauville--BMY strip
\[
        8<\frac{K_{Y_n}^2}{\chi(\OO_{Y_n})}\le9.
\]
The upper boundary is attained exactly at the three ball-quotient members: Hesse
at $n=3$, dual Hesse at $n=5$, and the complete quadrangle at $n=5$.  Indeed,
with the denominators displayed above,
\[
\begin{aligned}
9-\frac{K_{Y_n}^2}{\chi(\OO_{Y_n})}
&=\frac{9(n-3)^2}{21n^2-54n+37} &&\text{(Hesse)},\\
&=\frac{3(n-5)^2}{11n^2-30n+23} &&\text{(dual Hesse)},\\
&=\frac{3(n-5)^2}{7n^2-30n+35} &&\text{(complete quadrangle)}.
\end{aligned}
\]
On the other hand $K^2/\chi-8=\sig/\chi$, so positivity of the signature is
exactly what places the examples above the lower boundary of this strip.

Each family contains infinitely many pairwise non-isomorphic surfaces, since the
Chern numbers grow with $n$.  The Hesse family is the most economical from the
viewpoint of the Rees construction: four cubic sections already generate the
singular ideal.  The complete quadrangle is conceptually complementary: its
first ball-quotient member lies exactly on the Beauville boundary, while the
higher Kummer exponents move immediately into the strict non-existence region.

\section{Ulrich sheaves, known existence results, and scope}\label{sec:scope}

The conclusions above concern the specified very ample classes $H_n$; they do
not assert non-existence for every polarization on the same underlying surface.
This is exactly the fixed-embedding character of the ESW problem.  Indeed,
Coskun--Huizenga prove that for every smooth complex projective surface $X$ and
every ample divisor $H$, the pairs $(X,mH)$ admit rank-two Ulrich bundles for
$m\gg0$ \cite[Theorem~1.2]{CoskunHuizenga}.  Our examples show that an individual
very ample embedding can nevertheless have infinite Ulrich complexity.

The passage from bundles to the sheaf formulation is elementary but important
for the statement of the main theorem.

\begin{lemma}[Full-support Ulrich sheaves on smooth surfaces]\label{lem:full-support-bundle}
Let $(S,H)$ be a smooth polarized surface.  Every Ulrich sheaf on $S$ whose support
is all of $S$ is locally free.
\end{lemma}

\begin{proof}
An Ulrich sheaf is arithmetically Cohen--Macaulay, hence Cohen--Macaulay of full
dimension.  Its stalk at every point is therefore maximal Cohen--Macaulay over the
regular local ring $\OO_{S,x}$.  By Auslander--Buchsbaum it has projective dimension
zero and is free.  Equivalently, one may apply
\cite[Lemma~30.11.5, Tag~0B3K]{StacksCM}.
\end{proof}

Consequently, on our smooth $Y_n$, excluding Ulrich bundles excludes exactly the
full-support Ulrich sheaves in the standard scheme-theoretic ESW formulation
\cite[p.~543]{ESW}\cite[Proposition~4.2 and Problem~4.3]{HanselkaKummer}.

There is no conflict with standard positive existence theorems.  The projective
models $X_n$ are singular complete intersections, whereas the polarizations used
on the resolutions are $H_n=(\nu+1)A_n-E_n$; Ulrichness is not a birational
invariant.  Likewise, results transferring Ulrich bundles under blow-ups at
smooth points \cite{CasnatiKim,Secci} do not apply to the present resolutions,
whose centers are singular points and whose exceptional curves have positive
genus in the principal examples.

The relation with Beauville's discussion is geometric as well as numerical.
Ball quotients with $K=3L$ satisfy $L^2=\sig$.  The Hesse and dual-Hesse members at
$n=3$ and $n=5$ respectively admit very ample classes $L-A$ strictly below that
boundary, while the complete quadrangle at $n=5$ realizes the same numerical
boundary in higher Picard rank.  The infinite families show that this off-ray
mechanism is not confined to a single ball quotient but persists throughout
Kummer-cover families.

There is a distinct non-existence phenomenon in local commutative algebra.
Yhee gave the first counterexamples to the existence of Ulrich modules for
complete local domains, in every dimension at least two \cite{Yhee2023}.
Iyengar--Ma--Walker--Zhuang subsequently constructed two-dimensional
Cohen--Macaulay local rings with no Ulrich modules, including Gorenstein normal
domains \cite{IyengarMaWalkerZhuang}; complementary positive results for a large
class of two-dimensional local rings were obtained by Iyengar--Ma--Walker
\cite{IyengarMaWalker2025}.  These are local, singular phenomena and are distinct
from the smooth polarized projective question considered here.  For broader
background on Ulrich bundles and modules we refer to
\cite{CostaMiroRoigPonsLlopis}.

Recent work has also produced strong finite-rank obstructions.
Lopez--Raychaudhury exclude Ulrich bundles of rank at most three on broad classes
of complete intersections \cite{LopezRaychaudhury2026}, while
Lopez--Raychaudhury--Takahashi obtain a linear lower bound for the Ulrich
complexity of hypersurfaces in terms of their dimension
\cite{LopezRaychaudhuryTakahashi2026}.  These results provide quantitative or
bounded-rank non-existence, rather than examples of infinite Ulrich complexity.
General bidouble planes provide another nearby testing ground:
Caro--Cruz-Penagos--Troncoso determine Picard numbers for broad families and derive
constraints on Ulrich line bundles and Ulrich complexity
\cite{CaroCruzPenagosTroncoso}.  Their results do not produce all-ranks
non-existence of the type constructed here.

Recent papers still formulate the general Ulrich existence question as open
\cite{MiroRoig2026,Bordoni2025}.  We are not aware of previously published
smooth polarized varieties for which Ulrich bundles of every rank are excluded.
Accordingly, the priority claim made here is deliberately phrased ``to the best
of our knowledge'' and concerns fixed polarized pairs.

\section{Concluding perspective}

The paper has two main conclusions.  First, the fixed-embedding ESW existence
problem has a negative answer already for smooth complex projective surfaces.  The
Hesse ball quotient with the very ample class $H=4A-E$ satisfies
$H^2<\sig(Y)$ and therefore carries no full-support Ulrich sheaf for that embedding.
Equivalently, the Chow form of this embedding has no linear determinantal
representation arising from an Ulrich sheaf on the embedded surface through the ESW construction.  The
geometry explains how the example realizes Beauville's numerical strategy: on the
canonical ray $L=K/3$ one has $L^2=\sig$, while the extra N\'eron--Severi direction
allows the off-ray move $L\mapsto L-A$ without losing very ampleness.

Second, the counterexample is a systematic phenomenon rather than an isolated
polarization on one classical surface.  The Hesse construction gives infinitely many
pairwise non-isomorphic smooth ESW counterexamples, and the general Kummer mechanism
produces further infinite families.  Three pieces are responsible:
\begin{enumerate}
\item the rank-independent inequality $H^2\ge\sig(S)$ forced by every Ulrich bundle
      on a smooth polarized surface;
\item a Rees-algebra/Segre mechanism converting low-degree generation of the singular
      ideal into explicit very ample classes on Kummer resolutions;
\item a signature formula whose leading term grows like $n^{k-1}$, while the
      polarization square grows only like $n^{k-3}$.
\end{enumerate}
Consequently $H_n^2/\sig(Y_n)\to0$ in the Hesse, dual-Hesse, and complete-quadrangle
families: the violation of the Ulrich threshold becomes arbitrarily large in relative
terms.

Several questions are left open by this picture.  Is the necessary inequality
$H^2\ge\sig(S)$ close to sufficient under natural hypotheses?  Which regions of the
very ample cone of a fixed surface can have infinite Ulrich complexity?  Which
incidence hypergraphs yield the low-degree cover ideals needed by the Kummer
construction?  Finally, it would be interesting to find higher-dimensional numerical
obstructions playing the role of the surface inequality used here.

\section*{Acknowledgements}
The author thanks Arnaud Beauville for an illuminating e-mail exchange following his 2017 Beijing talk on Ulrich bundles, and the organizers of the 2017 conference in Beijing for the financial support that made the author's participation possible.

\section*{Statements and declarations}
\noindent\textbf{Funding.}
The author received no research funding for this work.

\medskip
\noindent\textbf{Competing interests.}
The author has no financial interests to disclose.

\medskip
\noindent\textbf{Use of generative AI.}
ChatGPT and Claude were used for exploratory computations, literature search and editorial assistance. All proofs are self-contained and the author is solely responsible for the final text.

\end{document}